\documentclass[reqno,11pt]{amsart}

\usepackage{amsmath,amsfonts,amssymb,amsthm,epsfig,amscd,}
\usepackage[]{fontenc}
\usepackage[latin1]{inputenc}
\usepackage{enumerate}
\usepackage[cmtip,all]{xy}
\usepackage{tabularx,multirow}
\usepackage{color}
\usepackage{enumitem}
\setlist[itemize]{itemsep=0ex,parsep=0ex,topsep=0ex,partopsep=0ex,leftmargin=4.5ex}
 \allowdisplaybreaks \numberwithin{equation}{section}

 \newtheorem{thm}{Theorem}[section]

\newtheorem{result}[thm]{Result}
 \theoremstyle{definition}
 
 \theoremstyle{remark}
 
\newtheorem{remark}[thm]{Remark}
 
 \numberwithin{equation}{section}

\def\geq{{\geqslant}}

\begin{document}

\title[On Fano schemes of linear spaces of general complete intersections]{On Fano schemes of linear spaces of general complete intersections}



\thanks{This collaboration has benefitted of funding from the MIUR Excellence Department Project awarded to the Department of Mathematics, 
University of Rome Tor Vergata (CUP: E83-C18000100006).}

\author{Francesco Bastianelli}
\address{Francesco Bastianelli, Dipartimento  di Matematica, Universit\`{a} degli Studi di Bari ``Aldo Moro", Via Edoardo Orabona 4, 70125 Bari -- Italy}
\email{francesco.bastianelli@uniba.it}

\author{Ciro Ciliberto}
\address{Ciro Ciliberto, Dipartimento  di Matematica, Universit\`{a} degli Studi di Roma ``Tor Vergata", Viale della Ricerca Scientifica 1, 00133 Roma -- Italy}
\email{cilibert@mat.uniroma2.it}

\author{Flaminio Flamini}
\address{Flaminio Flamini, Dipartimento  di Matematica, Universit\`{a} degli Studi di Roma ``Tor Vergata", Viale della Ricerca Scientifica 1, 00133 Roma -- Italy}
\email{flamini@mat.uniroma2.it}

\author{Paola Supino}
\address{Paola Supino, Dipartimento  di Matematica e Fisica, Universit\`{a} degli Studi ``Roma Tre", Largo S. L. Murialdo 1, 00146 Roma -- Italy}
\email{supino@mat.uniroma3.it}

\begin{abstract} We consider the {\em Fano scheme} $F_k(X)$ of $k$--dimensional linear  subspaces contained in a  complete intersection $X \subset \mathbb{P}^n$ of multi--degree $\underline{d} = (d_1, \ldots, d_s)$. Our main result is an extension of a result of Riedl and Yang concerning Fano schemes of lines on very general hypersurfaces: we consider the case when $X$ is a very general complete intersection and $\Pi_{i=1}^s d_i > 2$ and we find conditions on $n$, $\underline{d}$ and $k$ under which $F_k(X)$ does not contain either rational or elliptic curves.  At the end of the paper, we study the case $\Pi_{i=1}^s d_i = 2$. 
\end{abstract}

\maketitle

%%%%%%%%%%%%%%%%%%%%%%%%%%%%%%%%%%%%%%%%%%%%%%%%%
%%%%%%%%%%%%%%%%%%%%%%%%%%%%%%%%%%%%%%%%%%%%%%%%%
%%%%%%%%%%%%%%%%%%%%%%%%%%%%%%%%%%%%%%%%%%%%%%%%%

\section{Introduction} In this paper we are concerned with the {\em Fano scheme} $F_k(X)\subset \mathbb G(k,n)$, parameterizing $k$--dimensional linear subspaces contained in $X \subset \mathbb{P}^n$, when $X$ is a complete intersection of multi--degree $\underline{d} = (d_1, \ldots, d_s)$, with $1\leqslant s \leqslant n-2$. 
We will avoid the trivial case in which $X$ is a linear subspace, so that $\Pi_{i=1}^s\;d_i \geqslant 2$.

Our inspiration  has been the following result by Riedl and Yang concerning the case of hypersurfaces: 

\begin{thm}\label{res2} (cf.\;\cite[Thm.\;3.3]{RY}) If $X \subset \mathbb{P}^n$ is a very general hypersurface of degree $d$ such that $n \leqslant \frac{(d+1)(d+2)}{6}$, then $F_1(X)$ 
contains no rational curves. 
\end{thm}

This paper is devoted to generalize Riedl--Yang's result to complete intersections and to arbitrary $k\geq 1$:

\begin{thm}\label{thm:BCFS}  Let $X \subset \mathbb{P}^n$ be a very general complete intersection of multi--degree $\underline{d} = (d_1, \ldots, d_s)$, with 
$1\leqslant s \leqslant n-2$ and $\Pi_{i=1}^s d_i >2$. 

\noindent
Let $ 1 \leqslant k \leqslant n-s-1$ be an integer.
If 
\begin{equation}\label{eq:boundn}
n \leqslant k-1+  \frac{1}{k+2}\sum_{i=1}^s  {d_i+k+1 \choose k+1} , 
\end{equation} 
then $F_k(X)$ contains neither rational nor elliptic curves. 
\end{thm}
\noindent The proof  is  contained in Section \ref{S2}.  Section \ref{S3} concerns the quadric case $\Pi_{i=1}^s\;d_i = 2$.

We work over the complex field $\mathbb{C}$. As customary, the term {\em  ``general''} is used to denote a point which sits outside a union 
of finitely many proper closed subsets of an irreducible algebraic variety whereas the term {\em ``very general''} is used to denote a point sitting outside a countable union 
of proper closed subsets of an irreducible algebraic variety.
\bigskip

%%%%%%%%%%%%%%%%%%%%%%%%%%%%%%%%%%%%%%%%%%%%%%%%%
%%%%%%%%%%%%%%%%%%%%%%%%%%%%%%%%%%%%%%%%%%%%%%%%%
%%%%%%%%%%%%%%%%%%%%%%%%%%%%%%%%%%%%%%%%%%%%%%%%%

\section{Preliminaries}\label{S:Pre} 

Let $n \geqslant 3$, $1\leqslant s \leqslant n-2$ and $\underline{d} = (d_1, \ldots, d_s)$ be an $s$--tuple of positive integers such that $\Pi_{i=1}^s\;d_i \geqslant 2$. 
Let $S_{\underline{d}}:= \bigoplus_{i=1}^s H^{0}\left(\mathbb P^n,\mathcal{O}_{\mathbb{P}^{n}}(d_i) \right)$ and consider its Zariski open subset $S^*_{\underline{d}}:= \bigoplus_{i=1}^s \big ( H^{0}\left(\mathbb P^ n,\mathcal{O}_{\mathbb{P}^{n}}(d_i) \right)\setminus \{0\}\big )$. 
For any $u := (g_1,\ldots,g_s) \in S^*_{\underline{d}}$, let $X_u:=V (g_1,\ldots,g_s)\subset \mathbb{P}^n$ denote  the closed subscheme defined by the vanishing of the polynomials $g_1,\ldots,g_s$. 
When $u \in S^*_{\underline{d}}$ is general, $X_u$ is a smooth, irreducible variety of dimension $n-s \geqslant 2$, so that $S^*_{\underline{d}}$ contains an open dense subset parameterizing $s$--tuples $u$ such that $X_u$ is a smooth complete intersection. 

For any integer $1 \leqslant k \leqslant n-s-1$, consider the locus
$$
W_{\underline{d},k}:=\left\{\left.u\in S^*_{\underline{d}} \right| \; F_k(X_u) \neq \emptyset \; \right\} \subseteq S^*_{\underline{d}}$$and set
\[
t(n,k,\underline {d}):= (k+1)(n-k) - \sum_{i=1}^s \binom{d_i +k}{k} .
\] If no confusion arises, we will denote by $t$ the integer $t(n,k,\underline {d})$.  In this set--up, we remind the following results: 

\begin{result}\label{res1} (cf.,\;e.g.,\;\cite{Bo,DeMa,Langer,Morin,Pre} and \cite[Thm.\,22.14,\,p.\,294]{H}) Let $n,k,s$ and $\underline{d}=(d_1,\dots,d_s)$ be integers as above. 

\smallskip

\noindent
$(a)$ When $\Pi_{i=1}^s d_i >2$, one has the following situation.
\begin{itemize}
\item[(i)] For $t <0$,  $W_{\underline{d},k} \subsetneq S_{\underline{d}}^*$ so, for $u \in S_{\underline{d}}^*$ general,  $F_k(X_u) = \emptyset$.
\item[(ii)] For $t \geqslant 0$, $W_{\underline{d},k} = S_{\underline{d}}^*$ and, for $u \in S_{\underline{d}}^*$ general,  $F_k(X_u)$ is smooth with $\dim (F_k(X_u)) = t$ and it is  irreducible when  $t \geqslant 1$. 
\end{itemize}

\smallskip 

\noindent
$(b)$ When  $\Pi_{i=1}^s d_i = 2$, one has the following situation. 

\begin{itemize}
\item[(i)] For $ \lfloor \frac{n-s}{2} \rfloor < k \leqslant n-s-1$, $W_{\underline{d},k} \subsetneq S_{\underline{d}}^*$; more precisely, for $u \in S_{\underline{d}}^*$ such that $X_u$ is smooth, one has
$F_k(X_u) = \emptyset$.
\item[(ii)] For $1 \leqslant k \leqslant \lfloor \frac{n-s}{2} \rfloor$,  
$W_{\underline{d},k} = S_{\underline{d}}^*$ and, for $u \in S_{\underline{d}}^*$ such that $X_u$ is smooth, then  $F_k(X_u)$ is smooth, (equidimensional) with 
$\dim (F_k(X_u)) = t = (k+1)(n-s- \frac{3k}{2})$. Moreover $F_k(X_u)$ is irreducible unless $\dim (X_u) = n-s$ is even and $k= \frac{n-s}{2}$, in which case $F_k(X_u)$ consists of two disjoint irreducible components. In this case, each irreducible component of $F_k(X_u)$ is isomorphic to $F_{k-1}(X_u')$, where $X_u'$ is a general hyperplane section of $X_u$. 
\end{itemize}
\end{result}

\begin{result}\label{res1bis} (cf.,\;e.g.,\;\cite[Rem.\,3.2,\,(2)]{DeMa}) Let $n,k,s$ and $\underline{d}=(d_1,\dots,d_s)$ be integers as above with $\Pi_{i=1}^s d_i \geqslant 2$. 
When $u \in  S_{\underline{d}}^*$ is such that $X_u \subset \mathbb{P}^n$ is a smooth, irreducible complete intersection of dimension $n-s$ and 
$F_k(X_u)$ is not empty, smooth and of (expected) dimension $t$, the canonical bundle of  $F_k(X_u)$ is given by
\begin{equation}\label{eq:canonicoFk}
\omega_{F_k(X_u)} = {\mathcal O}_{F_k(X_u)} \left( -n - 1 + \sum_{i=1}^s \binom{d_i+k}{k+1}\right) 
\end{equation} 
where ${\mathcal O}_{F_k(X_u)} (1)$ is the hyperplane line bundle of  $F_k(X_u)$ in $\mathbb P^{{{n+1}\choose {k+1}}-1}$ via the Pl\"ucker embedding
of $\mathbb G(k,n)$.
\end{result} We will also need the following:

\begin{result}\label{res2bis} (cf.,\cite[Prop.\;3.5]{RY}) Let $n,m$ be positive integers with $m\leqslant n$. Let $B\subset {\mathbb G}(m,n)$ be an irreducible subvariety of codimension at least $\epsilon \geqslant 1$. Let $C\subset {\mathbb G}(m-1,n)$ be a non--empty subvariety satisfying the following condition: for every $(m-1)$-plane $c\in C$, if $b\in {\mathbb G}(m,n)$ is such that $c\subset b$, then $b\in B$. Then the codimension of $C$ in 
$ {\mathbb G}(m-1,n)$ is at least $\epsilon + 1$.
\end{result}

%%%%%%%%%%%%%%%%%%%%%%%%%%%%%%%%%%%%%%%%%%%%%%%%%%%%%%%%%%%%%%%%
%%%%%%%%%%%%%%%%%%%%%%%%%%%%%%%%%%%%%%%%%%%%%%%%%%%%%%%%%%%%%%%
%%%%%%%%%%%%%%%%%%%%%%%%%%%%%%%%%%%%%%%%%%%%%%%%%%%%%%%%%%%%%%%%

\section{The proof of Theorem \ref{thm:BCFS}}\label{S2} This section is  devoted to the proof of Theorem \ref{thm:BCFS}, which uses a strategy similar to the one  
in \cite[Thm.\,3.3]{RY}.

\begin{proof}[Proof of Theorem \ref{thm:BCFS}] Let $\mathbb{G}:= \mathbb{G}(k,n)$ be the Grassmannian of $k$-planes in $\mathbb P^n$ and 
$\mathbb{H} := \mathbb{H}_{\underline{d}, n}$ be the irreducible component of the Hilbert scheme whose general point parameterizes a smooth, irreducible 
complete intersection $X \subset \mathbb{P}^n$ of dimension $n - s \geqslant 2$ and multi--degree $\underline{d}$ ($\mathbb{H}$ is the image of an open dense subset of 
$S_{\underline{d}}^*$ via the classifying morphism). Let us consider the incidence correspondence  
$$
{\mathcal U}_{k,n,\underline{d}}:=\left\{\left.\left(\left[\Lambda\right], [X] \right) \in \mathbb{G} \times \mathbb{H}\right| \Lambda  \subset X\right\} \subset \mathbb{G} \times \mathbb{H}
$$with the two  projections $\mathbb{G} \stackrel{\pi_1}{\longleftarrow} {\mathcal U}_{k,n,\underline{d}} \stackrel{\pi_2}{\longrightarrow} \mathbb{H}$. 

Note that ${\mathcal U}_{k,n,\underline{d}}$ is smooth and irreducible, because the map $\pi_1$ is surjective and has smooth and irreducible fibres which are all isomorphic via the action of the group of projective transformations. 

Since $\Pi_{i=1}^s d_i >2$, from Result \ref{res1}--(a), if 
$t \leqslant 0$ then $F_k(X)$ is either empty or a zero--dimensional scheme, so in particular the statement holds true. 
We can therefore assume $t \geqslant 1$. In this case, from Result \ref{res1}--(a.ii), the map $\pi_2$ is dominant and the fibre over the general point $[X] \in \mathbb H$ is isomorphic to 
$F_k(X)$, hence it is smooth, irreducible of dimension $t$. Thus ${\mathcal U}_{k,n,\underline{d}}$ dominates  $\mathbb H$ via $\pi_2$ and 
\begin{equation}\label{eq:dimU}
\dim ({\mathcal U}_{k,n,\underline{d}}) = \dim (\mathbb{H}) + t.
\end{equation} 

Consider now 
$$
{\mathcal R}_{k, n,\underline{d}} := \left\{ ([\Lambda], [X]) \in {{\mathcal U}_{k,n,\underline{d}}} \left| \begin{array}{l} \text{there exists a rational curve}\\
 \text{in $F_k(X)$ containing $[\Lambda]$}  \end{array}\right.\right\}\subseteq {{\mathcal U}_{k,n,\underline{d}}}
$$
and similarly
$$
{\mathcal E}_{k, n,\underline{d}} := \left\{ ([\Lambda], [X]) \in {{\mathcal U}_{k,n,\underline{d}}} \left| \begin{array}{l} \text{there exists an elliptic curve}\\
 \text{in $F_k(X)$ containing $[\Lambda]$}  \end{array}\right.\right\}\subseteq {{\mathcal U}_{k,n,\underline{d}}}.
$$
Notice that both ${\mathcal R}_{k,n,\underline{d}}$ and ${\mathcal E}_{k, n,\underline{d}}$ are (at most) countable unions of irreducible locally closed subvarieties. 
Let us see this for $\mathcal R_{k,n,\underline d}$, the case of $\mathcal E_{k,n,\underline d}$ being analogous. 
Look at the morphism $\pi_2: \mathcal U_{k,n,\underline d}\longrightarrow \mathbb H$, whose fibre over a point $[X]\in \mathbb H$ is isomorphic to the Fano scheme $F_k(X)$.
Consider the locally closed subset $\mathcal H$ in the Hilbert scheme consisting of the set of points parameterizing irreducible rational curves in $\mathcal U_{k,n,\underline d}$ contained in fibres of $\pi_2$. Then $\mathcal H$, as well as the Hilbert scheme,  is a countable union of irreducible varieties.  Consider the incidence correspondence 
\[
\mathcal I=\left\{\left.\left(([\Lambda], [X]\right), [C])\in \mathcal U_{k,n,\underline d}\times \mathcal H\right| ([\Lambda], [X])\in C\right\}.
\]
Since $\mathcal H$ is a countable union of irreducible varieties, then also $\mathcal I$ is a countable union of irreducible varieties. Finally $\mathcal R_{k,n,\underline d}$ is the image of $\mathcal I$ via the projection on $\mathcal U_{k,n,\underline d}$, and therefore it is (at most) a countable union of irreducible varieties.

We point out that proving the assertion for very general complete intersections $X\subset \mathbb{P}^n$ of multi--degree $\underline{d}=(d_1,\dots,d_s)$ is equivalent to proving it for very general complete intersections $Y\subset \mathbb{P}^M$ of multi--degree 
$$\underline{d}^n=\underline{d}^n_M:=(\underbrace{1,\dots,1}_{M-n},d_1,\dots,d_s)$$ 
for some $M\geqslant n$.
In particular, condition \eqref{eq:boundn} gives equivalent inequalities and $t:=t(n,k,\underline{d})=t(M,k,\underline{d}^n_M)$.
Moreover, in the light of this fact, we can assume $d_1,\dots,d_s\geqslant 2$.

We claim that, under our assumptions, there exists $M\geqslant n$ such that both ${\mathcal R}_{k,M,\underline{d}^n}$ and ${\mathcal E}_{k,M,\underline{d}^n}$ have codimension at least $t+1$ in ${\mathcal U}_{k,M,\underline{d}^n}$, which proves the statement. 
Indeed, consider the case of  ${\mathcal R}_{k,M,\underline{d}^n}$ (the same reasoning applies to the case with ${\mathcal E}_{k,M,\underline{d}^n}$); if ${\rm codim}_{{\mathcal U}_{k,M,\underline{d}^n}} ({\mathcal R}_{k,M,\underline{d}^n}) \geqslant t+1$ then, by formula \eqref {eq:dimU}, we have that  $\dim ({\mathcal R}_{k, M,\underline{d}^n}) \leqslant \dim (\mathbb{H}_{\underline{d}^n, M}) -1$.
Thus $\dim (\pi_2 \left( {\mathcal R}_{k, M,\underline{d}^n}  \right)) \leqslant \dim (\mathbb{H}_{\underline{d}^n, M}) -1$, proving the assertion for very general complete intersections $Y\subset \mathbb{P}^M$ of multi--degree $\underline{d}^n$, and hence for very general complete intersections $X\subset \mathbb{P}^n$ of multi--degree $\underline{d}$. 

We are therefore left to showing that there exists $M\geqslant n$ such that both ${\mathcal R}_{k,M,\underline{d}^n}$ and ${\mathcal E}_{k, M,\underline{d}^n}$ have codimension at least $t+1$ in ${\mathcal U}_{k,M,\underline{d}^n}$. In what follows, we will focus on ${\mathcal R}_{k,M,\underline{d}^n}$ (the same arguments work for the case ${\mathcal E}_{k, M,\underline{d}^n}$).

If ${\mathcal R}_{k, n,\underline{d}} = \emptyset$, we are done.
So assume ${\mathcal R}_{k, n,\underline{d}} \neq \emptyset$ and take $([\Lambda_0], [X_0]) \in {\mathcal R}_{k, n,\underline{d}}$ a very general point in an irreducible component of ${\mathcal R}_{k, n,\underline{d}}$. 
We note that for any $M\geqslant n$, we can embed $X_0\subset \mathbb{P}^n$ into a $n$-plane of $\mathbb{P}^M$, and hence we can identify $X_0$ to a complete intersection in $\mathbb{P}^M$ of multi--degree $\underline{d}^n$, so that $([\Lambda_0], [X_0]) \in {\mathcal R}_{k, M,\underline{d}^n}$.
Therefore it suffices to find some $M\geqslant n$ and an irreducible subvariety $\mathcal F \subset {\mathcal U}_{k,M,\underline{d}^n}$ such that 
$$([\Lambda_0], [X_0]) \in \mathcal F \;\; {\rm and} \;\;   {\rm codim}_{\mathcal F} ({\mathcal R}_{k, M,\underline{d}^n} \cap \mathcal F) \geqslant t+1.$$
To do so, we start with the following remark: if $X\subset \mathbb{P}^n$ is a very general complete intersection of multi--degree $\underline{d}$, it follows from Results \ref{res1}--(a.ii) and \ref{res1bis} and formula \eqref{eq:canonicoFk} that $\omega_{F_k(X)}$ is ample if and only if 
$$n \leqslant \sum_{i=1}^s \binom{d_i+k}{k+1} - 2.$$
We set 
$$m = m(\underline{d}, k) := \sum_{i=1}^s \binom{d_i+k}{k+1} - 2$$
and, in view of the obvious equality ${d_i+k+1 \choose k+1} = {d_i+k \choose k+1} + {d_i+k \choose k}$, we notice that the hypothesis \eqref {eq:boundn} reads $m-n\geqslant t$. 
Hence we have  $m>n$. 

Let $Y' \subset \mathbb{P}^m$ be a very general complete intersection of multi--degree $\underline{d}$ and let $\Lambda' \subset Y'$ be a very general $k$-plane of $Y'$, i.e., $[\Lambda'] $ corresponds to a very general point in $F_k(Y')$. 
%By applying, if necessary, a projective transformation, we may suppose that $\Lambda' = \Lambda_0$.
By Result \ref{res1} and the fact that $m > n$, we have that $F_k(Y')$ is 
smooth, irreducible with $\dim (F_k(Y'))= (k+1)(m-k) - \sum_{i=1}^s \binom{d_i +k}{k} > t \geqslant 1$. 
Notice that there are no rational curves in $F_k(Y')$ through $[\Lambda']$ since $F_k(Y')$ is smooth and of general type, because $\omega_{F_k(Y')}=\mathcal{O}_{F_k(Y')}(1)$ by the choice of $m$, and $[\Lambda']$ is very general on $F_k(Y')$.
In other words $([\Lambda'], [Y']) \in {\mathcal U}_{k,m,\underline{d}} \setminus {\mathcal R}_{k,m,\underline{d}}$. %where ${\mathcal U}_{k, m,\underline{d}}$ and ${\mathcal R}_{k, m,\underline{d}}$ are defined as above  for $\mathbb{P}^m$ instead of $\mathbb{P}^n$. 

Take now  an integer $M \gg m > n$ and let $Y'' \subset {\mathbb P}^M$ be a smooth, complete intersection of multi--degree $\underline{d}$, containing a $k$-plane $\Lambda''$, such that $X_0$ is a $n$-plane section of $Y''$, $Y'$ is a $m$-plane section of $Y''$, and $\Lambda''=\Lambda_0=\Lambda'$. 

For any integer $r \geqslant n$, let $Z_r \subset {\mathbb G}(r,M)$ be the variety of $r$-planes in $\mathbb{P}^M$ containing $\Lambda_0$ and let $Z'_r \subseteq Z_r$ be the subset of those $r$-planes $\Lambda \in Z_r$ such that the Fano scheme $F_k(Y'' \cap \Lambda)$ of the complete intersection $Y'' \cap \Lambda \subset \mathbb{P}^M$ contains a rational curve through the point $[ \Lambda_0] \in F_k(Y'' \cap \Lambda)$. Note that $Z_r$ is isomorphic to $\mathbb G(r-k-1, M-k-1)$. 

For any integer $r \geqslant n$, we define the morphism $\phi_r\colon Z_r\longrightarrow {\mathcal U}_{k,M,\underline{d}^r}$ sending the $r$-plane $\Lambda$ to  $([\Lambda_0],[\Lambda\cap Y''])\in {\mathcal U}_{k,M,\underline{d}^r}$. 
It is clear that $\phi_r$ maps $Z_r$ isomorphically onto its image $\mathcal F_r:=\phi_r(Z_r)$, where the inverse map sends a pair $([\Lambda_0],[Y])\in \mathcal F_r$ to the point $[\Lambda]\in Z_r$ corresponding to the linear span $\Lambda=\langle Y\rangle$ (recall that we assumed $d_1,\dots,d_s\geqslant 2$). 
Then the image of $Z'_r$ in ${\mathcal U}_{k,M,\underline{d}^r}$ is $\mathcal F_r \cap {\mathcal R}_{k,M,\underline{d}^r}$. 
We will set $\mathcal F:=\mathcal F_n$. 

By construction of the pair $([\Lambda_0], [Y']) \in {\mathcal U}_{k, M,\underline{d}^m} \setminus {\mathcal R}_{k, M,\underline{d}^m}$, we have that 
${\rm codim}_{\mathcal F_m} (\mathcal F_m \cap {\mathcal R}_{k,M,\underline{d}^m}) = \epsilon \geqslant 1$. 
Now we apply Result \ref {res2bis}, to $\mathcal F_m\cong \mathbb G(r-k-1, M-k-1)$, $B=\mathcal F_m \cap {\mathcal R}_{k,M,\underline{d}^m}$ and $C=\mathcal F_{m-1} \cap {\mathcal R}_{k,M,\underline{d}^{m-1}}$, because clearly $B$ and $C$ verify the condition stated there. 
We deduce that ${\rm codim}_{\mathcal F_{m-1}} (\mathcal F_{m-1} \cap {\mathcal R}_{k,M,\underline{d}^{m-1}}) \geqslant \epsilon+1$. 
Iterating this argument we see that ${\rm codim}_{\mathcal F} (\mathcal F \cap {\mathcal R}_{k,M,\underline{d}^n}) \geqslant m-n+ 1$. 
Now, as we already noticed,  $m - n + 1 \geqslant t+1$ is equivalent to the condition \eqref {eq:boundn}, hence we have ${\rm codim}_{\mathcal F} (\mathcal F \cap {\mathcal R}_{k,M,\underline{d}^n}) \geqslant t+1$,  as wanted. 
This completes the proof. \end{proof}

\begin{remark}\label{rem:pap} Let $X\subset \mathbb P^n$ be a general complete intersection of multi--degree $\underline{d} = (d_1, \ldots, d_s)$. If 
$\sum_{i=1}^s \binom{d_i+k}{k+1}\leqslant n,$ then,  by \eqref {eq:canonicoFk},
$F_k(X)$ is a smooth Fano variety and therefore by  \cite{KMM} it is \emph{rationally connected}, i.e., there is a rational curve passing through two general points of it. 

\end{remark}

%%%%%%%%%%%%%%%%%%%%%%%%%%%%%%%%%%%%%%%%%%%%%%%%%
%%%%%%%%%%%%%%%%%%%%%%%%%%%%%%%%%%%%%%%%%%%%%%%%%
%%%%%%%%%%%%%%%%%%%%%%%%%%%%%%%%%%%%%%%%%%%%%%%%%

\section{The quadric case}\label{S3}

Here we prove the following result:

\begin{thm}\label{thm:quadricstrong}  Let $X \subset \mathbb{P}^n$ be a smooth complete intersection of multi--degree $\underline{d} = (d_1, \ldots, d_s)$, with 
$1\leqslant s \leqslant n-2$ and $\Pi_{i=1}^s d_i = 2$. Let $k$ be an integer such that $ 1 \leqslant k \leqslant n-s-1$. Then:
\begin{itemize}
\item[(i)] for $ \lfloor \frac{n-s}{2} \rfloor < k \leqslant n-s-1$, $F_k(X)$ is empty, whereas 

\item[(ii)] for $ 1 \leqslant k \leqslant \lfloor \frac{n-s}{2} \rfloor$, the single component or both components of $F_k(X)$ (see Result \ref{res1}--(b.ii)) are rationally connected. 
\end{itemize}
\end{thm}

\begin{proof} Since  $\Pi_{i=1}^s d_i = 2$,
%, up to reordering the degrees $d_i$'s, we can assume that $\underline{d} = (2, 1, \ldots, 1)$. Therefore, 
we have that $X \subset \mathbb{P}^{n-s+1}$ is a smooth quadric hypersurface. From Result \ref{res1}--(b), if $\lfloor \frac{n-s}{2} \rfloor < k \leqslant n-s-1$, $F_k(X)$ is empty. 

Next we assume $1 \leqslant k \leqslant \lfloor \frac{n-s}{2} \rfloor$. 
%From Result \ref{res1}--(b.ii), if $\mathcal U_{k,n-s+1, 2} \subset \mathbb{G}(k, n-s+1) \times |{\mathcal O}_{\mathbb{P}^{n-s+1}}(2)|$ denotes the  incidence correspondence as in the proof of Theorem \ref{thm:BCFS}, then the projection $\pi_2: \mathcal U_{n-s+1, 2} \to |{\mathcal O}_{\mathbb{P}^{n-s+1}}(2)|$ to the second factor is dominant and its general fibre is smooth, of dimension $t = (k+1)(n-s-\frac{3k}{2}) \geqslant 1$. 
From Result \ref{res1bis} and formula \eqref{eq:canonicoFk}, we have 
$$\omega_{F_k(X)} = {\mathcal O}_{F_k(X)} \left( -n+s+k \right).$$Since $k \leqslant  \frac{n-s}{2}$, then $-n+s+k \leqslant -1$ therefore the single component or both components of $F_k(X)$ (see Result \ref{res1}--(b.ii)) are smooth Fano
varieties, hence they are rationally connected by \cite{KMM}. 
\end{proof}

%%%%%%%%%%%%%%%%%%%%%%%%%%%%%%%%%%%%%%%%%%%%%%%%%%%%%%%%%%%%%%%%
%%%%%%%%%%%%%%%%%%%%%%%%%%%%%%%%%%%%%%%%%%%%%%%%%%%%%%%%%%%%%%%
%%%%%%%%%%%%%%%%%%%%%%%%%%%%%%%%%%%%%%%%%%%%%%%%%%%%%%%%%%%%%%%%

\subsection*{Acknowledgment} The authors would like to thank Lawrence Ein for having pointed out the paper \cite{RY}. 

%%%%%%%%%%%%%%%%%%%%%%%%%%%%%%%%%%%%%%%%%%%%%%%%%%%%%%%%%%%%%%%%
%%%%%%%%%%%%%%%%%%%%%%%%%%%%%%%%%%%%%%%%%%%%%%%%%%%%%%%%%%%%%%%
%%%%%%%%%%%%%%%%%%%%%%%%%%%%%%%%%%%%%%%%%%%%%%%%%%%%%%%%%%%%%%%%

\end{document}